\documentclass[12pt,a4paper]{amsart}
\usepackage{amsfonts}
\numberwithin{equation}{section}

\theoremstyle{plain}
\newtheorem{Th}{Theorem}[section]
\newtheorem{Lemma}[Th]{Lemma}
\newtheorem{Cor}[Th]{Corollary}
\newtheorem{St}[Th]{Statement}

 \theoremstyle{definition}
\newtheorem{Def}[Th]{Definition}
\newtheorem{Conj}[Th]{Conjecture}

\newtheorem{Rem}[Th]{Remark}
\newtheorem{Ex}[Th]{Example}

\newtheorem{?}[Th]{Problem}

\newcommand{\tous}[1]{} 
\newtheorem{Prob}[Th]{Problem}

\newcommand{\norm}[1]{\left\| {#1}\right\| }

 \newcommand{\zj}[1]{\left({#1}\right)}

\newcommand\eq{\eqref}

 \newcommand\C{\mathbin{\overline *}}
\newcommand\f{\mathcal F}

\newcommand{\setZ}{{\mathbb Z}}
\newcommand{\setR}{{\mathbb R}}

\newcommand\1[1]{{ \bf 1}_{#1}}

\begin{document}

\title{An analytic approach to cardinalities of sumsets}                   

\author{D\'avid Matolcsi}
\email{matolcsidavid@gmail.com}
\author{Imre Z. Ruzsa}
\address{Alfr\'ed R\'enyi Institute of Mathematics\\
     Budapest, Pf. 127\\
     H-1364 Hungary
}
\email{ruzsa@renyi.hu}

\begin{abstract}
Let $d$ be a positive integer and $U \subset \mathbb{Z}^d$ finite. We study
$$\beta(U)  : = \inf_{\substack{A , B \neq \emptyset \\ \text{finite}}} \frac{|A+B+U|}{|A|^{1/2}{|B|^{1/2}}},$$
and other related quantities. We employ tensorization, which is not available for the doubling constant, $|U+U|/|U|$. For instance, we show
$$\beta(U) = |U|,$$ whenever $U$ is a subset of $\{0,1\}^d$. Our methods parallel those used for the Pr\' ekopa-Leindler inequality, an integral variant of the Brunn-Minkowski inequality. 
\end{abstract}

\author{George Shakan}
\address{SP and GS: Mathematical Institute, University of Oxford, Andrew Wiles Building, Radcliffe Observatory Quarter, Woodstock Road, Oxford, OX2 6GG, UK.}
\email{george.shakan@gmail.com}
\author{Dmitrii Zhelezov}
\address{Alfr\'ed R\'enyi Institute of Mathematics\\
     Budapest, Pf. 127\\
     H-1364 Hungary
}
\email{dzhelezov@gmail.com}
     \maketitle

     \tous{
I like the title (George)}

     \section{Introduction}

The aim of this study is to understand the nature of structures in $\setZ^d$, the presence of which implies that 
the sumset must be large. The archetype is Freiman's theorem that if a set $A \subset \setZ^d$ is proper $d$-dimensional,
then
\begin{equation}\label{Freiman}
|A+A| \geq (d+1) |A| - \binom{d+1}{2} . \end{equation}
The assumption on dimension can be expressed as $S_d \subset A$ for a $d$-dimensional simplex $S_d$. In general, the
 \emph{induced doubling} of a set $U$ is the quantity
   \[  \inf_{A\supset U} \frac{ |A+A|}{|A|} ; \]  
our main aim is to give lower estimates for it and related quantities. 
Applications for the sum-product problem, related to the work of \cite{BC}, will be the subject of another paper.

While our main interest is in  $\setZ^d$, we shall mostly formulate our results for general, typically torsion-free
commutative groups. Since we work with finite sets and a finitely generated torsion-free group is isomorphic to some
 $\setZ^d$, it is not more general, but we rarely need the coordinates.

In  the first part we work with sets, in the second part we study a weighted version which will
be necessary for the proof of the main results. By introducing a weighted analog, we will be able to use {\em tensorization}: that is we prove a $d$ dimensional inequality by induction on dimension alongside a two point inequality. This is a method commonly used in analysis, for instance in the Pr\'ekopa-Leindler inequality \cite{Prekopa} and Beckner's inequality \cite{Beck}. We discuss this more below, but also invite the reader to the excellent survey paper of Gardner \cite{Gardner}

\vskip 1cm
\centerline{ \large{ \bf Part I: sets}}

\section{Main results}

Let $U$ be a finite set in a  commutative group $G$. We modify the above definition of induced doubling to use sums of different sets,
which are often better behaved.

\begin{Def}[Induced doubling] The \emph{induced doublings} of $U$ are the quantities
   \[ \alpha(U) = \inf_{A\supset U, B\supset U} \frac{ |A+B|}{ \sqrt{|A| |B|}} , \]   
the (unrestricted) \emph{induced doubling};
\[ \alpha'(U) = \inf_{A\supset U, B\supset U, \  |A| = |B|} \frac{ |A+B|}{|A|} , \]   
the  \emph{isometric induced doubling};
\[ \alpha''(U) = \inf_{A\supset U} \frac{ |A+A|}{|A|} , \]   
      the \emph{isomeric induced doubling}.
       \end{Def}

       \begin{Conj}
         In the above definitions, the infimum is a minimum.
       \end{Conj}

       \tous{
         Easy in 1 dimension; follows from Freiman's $2^d$-theorem for $\alpha''$.}       

We rarely can estimate induced doubling directly, typically it will be through a related quantity involving the sum of three sets.

\begin{Def}[Induced tripling $\beta$] The \emph{triplings} of $U$ are the quantities
   \[ \beta(U) = \inf_{A,B}  \frac{ |A+B+U|}{\sqrt{|A| |B|}} , \]         
the \emph{(unrestricted) tripling};
\[ \beta'(U) = \inf_{A,B ,\  |A| = |B|}     \frac{ |A+B+U|}{\sqrt{|A| |B|}} , \]         
the \emph{isometric tripling};
\[ \beta''(U) = \inf_{A}  \frac{ |A+A+U|}{{|A|}} , \]         
the \emph{isomeric tripling}.
       \end{Def}

These infima may and may not be minima. Estimates for $\beta$ yield an estimate weaker than the obvious $\max ( |A|, |B|) $ when the sizes of
$A$ and $B$ are rather different. We consider an asymmetric version as follows.

\begin{Def}[Asymmetric induced tripling $\beta_p$]
  For $1 < p < \infty$ we put
 \[ \beta_p(U) = \inf_{A,B}  \frac{ |A+B+U|}{|A|^{1/p} |B|^{1-1/p}} . \]     
   \end{Def}
Thus $\beta(U) = \beta_{2}(U)$.  We shall estimate these quantities rather preciesly for sets contained in \emph{quasicubes}, which we define recursively
as follows.

\begin{Def}[Quasicubes]\label{Def:quasicube}
  A 0-dimensional quasicube is any singleton.

Let $U$ be a finite set in a commutative group $G$. We say that $U$ is a $d$-dimensional quasicube,
if there is a proper subgroup $G'$ such that  $U$ is contained
in two distinct cosets, say $G'+x$ and $G'+y$,  both $U\cap(G'+x)$ and  $U\cap(G'+y)$ are
$d-1$ -dimensional quasicubes and $x-y$ is of infinite order in the factor-group $G/G'$.
\end{Def}

For instance, in $\setZ^2$ any four points that lie on two distinct parallel lines (i.e. a trapezoid) form a quasi-cube. A $d$-dimensional quasicube has $2^d$ elements, and its dimension is indeed $d$ in according to the following definition.

\begin{Def}[Set dimension]
  Let $A$ be a finite set in a commutative group $G$. Let $H$ be the subgroup generated by $A-A$, that is, the
smallest group $H$ with the property that $A$ lies in a single coset of $H$. As a finitely generated group, $H$ is isomorphic
to some $H' \times \setZ^d$, where $H'$ is a torsion group. We call $d = \dim A$ the dimension of $A$.
\end{Def}

The central result of the present paper is that subsets of quasicubes induce large additive doubling and tripling. Indeed much of what we prove was known for cubes in \cite{GreenTao}, but their geometric methods do no extend to quasi-cubes (or subsets of quasi-cubes). 
\begin{Th}[Subsets of quasicubes have maximal $\beta$] \label{thm:quasicube_beta}
  Let $U$ be a $d$-dimensional quasicube in any commutative group. For every $V\subset U$ we have
\[ \beta(V) = |V|,  \ \ \alpha(V) \geq |V|^{1/2} . \]
In particular
  \[ \beta(U) = 2^d, \  \ \alpha(U) \geq 2^{d/2} . \]
\end{Th}
A short streamlined self-contained proof of Theorem~\ref{thm:quasicube_beta} can also be found in \cite{GMRSZ}.

The main innovation of the tripling $\beta$ is that it allows one to efficiently account for the additive expansion of the lower dimensional subsets (fibers) of $U$ in a recursive fashion. The core estimate is Theorem~\ref{Th:gamma2point}, where we show that a certain functional is minimized by geometric progressions. 

In comparison, the authors of \cite{BC} implicitly analysed a quantity similar to $\alpha$ and had to resort to multi-scale dyadic pigeonholing leading to a significantly worse estimate. In particular, such an analysis would give non-trivial bounds only for well-balanced quasicubes with all the lower-dimensional fibers being of comparable size. 

As a corollary of Theorem~\ref{thm:quasicube_beta}, it follows that iterated sumsets of quasicube sumsets grow logarithmically, which is essentially sharp.
\begin{Cor}[Quasi-cubes have large iterated sumset]
    Let $U$ be a $d$-dimensional quasicube in any commutative group. For every $V\subset U$ and $k\geq2$ we have
    $$
    |(2^k-1)V| \geq |V|^k.
    $$
\end{Cor}
\begin{proof}
    The base case $k=2$ follows from the definition of $\beta$ and Theorem~\ref{thm:quasicube_beta}.
    For larger $k$, one has
    $$
    |(2^{k}-1)V| = |(2^{k-1}-1)V + (2^{k-1}-1)V + V| \geq |(2^{k-1}-1)V|\beta(V) \geq |V|^k.
    $$
\end{proof}

The trivial bound
\begin{equation} \label{eq:trivial_bound}
    \beta(U) \leq \min (|U|, 2^d)
\end{equation} 
holds for any set $U$ of dimension $d$, so the induced tripling (i.e. $\beta$) of quasicube subsets is as large as it gets. We conjecture that this holds for a larger class of sets.

\begin{Conj}[Log-span conjecture] \label{conj:span}
  Let $V$ be a finite set with the property that for any $k \leq \dim V$ any $k$-dimensional subset of $V$ has at most $2^k$ elements. Then
   \begin{equation}
       \beta(V) = |V|
   \end{equation} and in particular 
   \[ \alpha(V) \geq |V|^{1/2} . \]
\end{Conj}

We conjecture that in fact $\beta$ is determined by the linear dependence matroid of the set in question, in the following sense.

\begin{Conj}[Linear matroid conjecture] \label{conj:matroid}
  Let $U, V$ be finite sets of equal cardinality in any group, $\varphi: U \to V$ a bijection.
If for every $U'\subset U$ we have $\dim \varphi(U') \leq \dim U'$, then $\beta(V) \leq \beta(U)$.
In particular, if always $\dim \varphi(U') = \dim U'$, then $\beta(V) = \beta(U)$.
\end{Conj}

Note that Conjecture~\ref{conj:span} would follow quickly from Conjecture~\ref{conj:matroid} and Theorem~\ref{thm:quasicube_beta}.

\begin{Th}[Discrete Pr\'ekopa-Leindler for quasi-cubes]\label{mainthm}
  Fix $1 < p < \infty$ and let $q$ be the conjugate exponent defined via 
  $$1/p + 1/q = 1.$$
  Let $U$ be a $d$-dimensional quasicube in any commutative group and $V \subset U$. We have
  \[ \beta_p(V) \geq c_p^d|V|  , \ \ \ \text{where} \  c_p = \frac{p^{1/p} q^{1/q}}{2} \leq 1  . \]

\end{Th}

The flexibility of choosing $p$ allows us to deduce the following discrete Brunn-Minkowski inequality. 

\begin{Cor}[Discrete Brunn-Minkowski for quasi-cubes]\label{BM} 
  Let $U$ be a subset of a $d$-dimensional quasi-cube in any commutative group. For any finite sets $A,B$
  we have
   \[ |A+B+U|^{1/d} \geq \frac{|U|}{2^d} \left(|A|^{1/d} +|B|^{1/d} \right) . \]
\end{Cor}

Note if $U$ is a quasi-cube, then $|U| = 2^d$, and we obtain
$$|A+B+U| \geq |A|^{1/d} + |B|^{1/d}.$$
This result was obtained for cubes by Green and Tao \cite[Lemma~2.4]{GreenTao}. Their methods, which rely on the continuous Brunn-Minkowski inequality, seem to not generalize to quasi-cubes.  We remark that our results are somewhat in a similar spirit to that of  \cite[Section 5]{Bour}, where lower bounds for sumsets of subsets of $\{0 , \ldots , M-1\}^d$ are provided. 

\begin{proof}
  Apply the inequality from Theorem~\ref{mainthm},
   \[|A+B+U| \geq \frac{|U|c_p^d}{2^d}  |A|^{1/p} |B|^{1/q} \]
   with the optimal choice of $p$ which is defined by
    \[ 1/p = \frac{ |A|^{1/d}}{ |A|^{1/d} +|B|^{1/d}}. \]
\end{proof}

Theorem~\ref{mainthm} can be viewed as a discrete Pr\'ekopa-Leindler inequality, which we recall (see also \cite[Theorem~4.2]{Gardner}).

\begin{Th}[Pr\'ekopa-Leindler \cite{Prekopa}] \label{PL}
Let $0 < \lambda < 1$ and 
$$g,h,F: \mathbb{R}^d \to \mathbb{R},$$ be non-negative measurable functions satisfying for all $x,y\in\mathbb{R}^d$
$$F((1-\lambda) x + \lambda y) \geq f(x) g(y).$$
Then
$$
\int F \geq ||f||_p ||g||_q,
$$
with $p = 1/\lambda$ and $1/p + 1/q = 1$.
\end{Th}

Note that Theorem~\ref{PL} can be used to deduce the Brunn-Minkowski inequality, in a similar manner to Corollary~\ref{BM}. Theorem~\ref{mainthm} can be intepreted to be a discrete analog of Theorem~\ref{PL}. 

\section{Inequalities between doublings and triplings}\label{23ineq}

We conjecture that the defined six quantities are actually only two, and
connected by simple inequalities.

\begin{Conj}[Doubling-tripling conjecture]\label{dct}
  For every finite set $U$ in any commutative group we have
   \[ \alpha(U) = \alpha'(U) = \alpha''(U) \leq \beta(U) = \beta'(U) = \beta''(U) \leq \alpha(U)^2 .\]
\end{Conj}

We list some properties.
\begin{St}[Basic Inequalities]\label{basicprops}
  Let $V$ be a finite set in a commutative group, $G$, $|V| =n$, $\dim V =d$. We have
   \[ \alpha(V) \leq \alpha'(V) \leq \alpha''(V)
   \begin{cases}
     < 2^d, \cr
     \leq (n+1)/2,
   \end{cases}   \]
  \[   d+1 \leq\beta(V) \leq \beta'(V) \leq \beta''(V)
\begin{cases}
     \leq  2^d, \cr
     \leq n.
   \end{cases}   \]
\end{St}

\begin{proof}
The inequalities 
$$  \alpha(V) \leq \alpha'(V) \leq \alpha''(V), \ \ \ \beta(V) \leq \beta'(V) \leq \beta''(V),$$
follow immediately from the definitions.
Taking $A= V$ in the definition of $\alpha''(V)$, we have
$$\alpha''(V) \leq \frac{|V+V|}{|V|} \leq \frac{1}{|V|} \binom{|V|+1}{2} = \frac{n+1}{2}.$$
 Taking $A = \{0\}$ in the definition of $\beta''(V)$, we find that 
 $$\beta''(V) \leq |V| =n.$$
Since $V$ has dimension $d$, we may assume
$$V \subset H' \times \mathbb{Z}^d.$$
Thus for large enough $N$, we have $V\subset A$, where 
$$A : = H' \times \{-N , \ldots , N\}^d.$$
Since 
$$\frac{|A+A|}{|A|} \to 2^d \ \ \ \text{as}  \ N \to \infty,$$
 we find that $\alpha''(V) \leq 2^d$.
 Also, 
 $$A+A+V \subset H' \times \{-2N - \max_{v \in V} |v|_{\infty} , \ldots ,  2N+\max_{v \in V} |v|_{\infty}\}^d,$$
 and so 
 $$\beta''(V) \leq \frac{|A+A+V|}{|A|} \to 2^d \ \ \ \text{as} \ N\to \infty.$$

Note that $\beta(V) \geq d+1$ follows from the more general Theorem~\ref{thm:quasicube_beta} which we prove later. 

\end{proof}

\tous{For $\alpha''$, I (George) do not see how to circumvent torsion. Even if there is no torsion, proving this seems to be quite tricky, requiring some delicate analysis of Gardner-Gronchi. I can do it for $d=2$, and very likely if $|A|$ and $|B|$ are large enough, the problem is checking when $|A|$ and $|B|$ are relatively small (compared to $d$).}

\begin{St}[Basic Inequalities II]
  For every finite set $U$ in any commutative group we have
\begin{equation}\label{alfabeta}
\alpha(U) \leq \beta(U), \ \alpha'(U) \leq 4\beta'(U),   \ \alpha''(U) \leq 3\beta''(U),  \end{equation}
\begin{equation}\label{betaalfa}
  \beta(U) \leq \alpha(U)^2,  \ \beta''(U) \leq \alpha(U)^3,  \end{equation}
 \begin{equation}\label{aabb}
 \alpha''(U) \leq \alpha(U)^2, \ \beta''(U) \leq \beta''(2U) \leq \beta(U)^ 2. \end{equation}
\end{St}

\begin{proof}
We may assume $U \subset G = H' \times \mathbb{Z}^d$. If $d=0$, then all
   \[ 1 = \alpha(U) = \alpha'(U) = \alpha''(U) = \beta(U) = \beta'(U) = \beta''(U) ,\] 
so we may assume $d \geq 1$. 

We first show the second  inequality in \eq{alfabeta}. Let $A,B$ be such that $|A|= |B|$ and let $k$ be a large integer. Since $d\geq 1$, we may choose a $x\in G$
be such that the sets 
$$A+x,  \ldots, A+kx,$$ are disjoint,
and 
 $$B+x \ldots, B+kx,$$
 are also disjoint.   Put
 \[ A' = U \cup \bigcup_{i=1}^k (A+ix), \   B' = U \cup \bigcup_{i=1}^k (B+ix) . \]
 These sets satisfy 
 $$U\subset A',B' \ \ \ \text{and} \ \ \  |A'|= |B'| \geq k |A|.$$ We have
\begin{equation}\label{uniok}
 A'+B' = (U+U) \cup \bigcup_{i=1}^k (A+U+ix)  \cup \bigcup_{i=1}^k (B+U+ix)  \cup \bigcup_{i=2}^{2k} (A+B+ix) .  \end{equation}
  As $|A+U|, |B+U|, |A+B|$ are all smaller than $|A+B+U|$, we find
  \[ |  A'+B'| \leq |U+U| + (4k-1)|A+B+U|.    \]
  Thus
   \[  \frac{ |  A'+B'|}{ |A'|} \leq \frac{ |U+U|}{ k |A|} + \zj{4-\frac{1}{k}} \frac{|A+B+U|}{|A|}.\]
As this is true for any $A$ and $B$ with $|A| = |B|$, we find
$$\alpha'(U) \leq \frac{|U+U|}{k} + 4 \beta(U),$$
and the second statement of \eqref{alfabeta} follows from letting $k \to \infty$.

The proof of the third statement of \eqref{alfabeta}, that is  $\alpha''(U) \leq 3\beta''(U)$, proceeds similarly. The only difference is we take $A=B$ so some parts of \eq{uniok} coincide
and the 4 is reduced to 3.

We could use the same approach to show $\alpha(U) \leq 4\beta(U)$. The proof below due to Thomas Bloom allows one to get rid of the factor. We need the following from \cite{Petridis}.

\begin{Lemma}[Petridis]\label{Pet} 
Let $X,Y,Z$ be finite subsets of a commutative group with the property that for all $X' \subset X$, we have 
$$\frac{|X'+Y|}{|X'|} \geq \frac{|X+Y|}{|X|}.$$
Then 
$$|X+Y+Z||X|\leq |X+Y||X+Z|.$$
\end{Lemma}

Let $A,B \subset G$ be finite with the property that for any $A' \subset A$ and $B'\subset B$
\begin{equation} \label{eq:minimal_prop}
\frac{|A+B+V|}{|A|^{1/2}|B|^{1/2}} \leq \frac{|A'+B+V|}{|A'|^{1/2}|B|^{1/2}}, \ \ \ \frac{|A+B+V|}{|A|^{1/2}|B|^{1/2}} \leq \frac{|A+B'+V|}{|A|^{1/2}|B'|^{1/2}}.
\end{equation}
By a standard limiting argument we may assume WLOG that the infimum in the definition of $\beta(V)$ is taken over $A, B$ satisfying (\ref{eq:minimal_prop}). This implies in particular that for any $A' \subset A$
$$
\frac{|A+B+V|}{|A|} \leq \frac{|A'+B+V|}{|A'|}.
$$
Applying Lemma~\ref{Pet} with $X  = A$, $Y  = B+V$ and $Z =V$, we conclude
$$|A+B+V+V| |A| \leq |A+B+V| |A+V|,$$ 
and rearranging gives
$$\frac{|A+V+B+V|}{|A+V|^{1/2} |B+V|^{1/2}} \left( \frac{|B+V|^{1/2}|A|^{1/2}}{|B|^{1/2}|A+V|^{1/2}}\right) \leq \frac{|A+B+V|}{|A|^{1/2}{|B|^{1/2}}}.$$ Applying with the roles of $A$ and $B$ swapped, we also find  
$$\frac{|A+V+B+V|}{|A+V|^{1/2} |B+V|^{1/2}} \left( \frac{|A+V|^{1/2}|B|^{1/2}}{|A|^{1/2}|B+V|^{1/2}}\right) \leq \frac{|A+B+V|}{|A|^{1/2}{|B|^{1/2}}},$$
and so
$$\frac{|A+V+B+V|}{|A+V|^{1/2} |B+V|^{1/2}} \leq \frac{|A+B+V|}{|A|^{1/2}{|B|^{1/2}}}.$$
Thus we conclude that 
$$\alpha(V) \leq \beta(V).$$

For \eq{betaalfa} and \eq{aabb}  we need Pl\"unnecke's inequality.

\begin{Lemma}[Pl\"unnecke]\label{Plu}
Let $X$ and $Y$ be subsets of a commutative group. Let $k$ be a positive integer and  $|X+Y| = c|X|$. Then there is a $X'\subset X$ such that 
$$|X'+kY| \leq c^k |X'| .$$
\end{Lemma}
In particular,
\begin{Lemma}[Pl\"unnecke] \label{Plu_alt}
  Let $X, Y$ be finite sets of an additive group. Then there is $X' \subset X$ such that
  $$
  |Y+Y| \leq |X' + Y + Y| \leq |X'| \frac{|X+Y|^2}{|X|^2} \leq \frac{|X+Y|^2}{|X|}
  $$
\end{Lemma}

We now proceed to \eqref{betaalfa}. Let $A, B$ be any sets containing $U$.  After swapping the roles of $A$ and $B$, we may suppose $|A| \geq |B| $. By Lemma~\ref{Plu}, there is an $A'\subset A$ such that
 \[ |A'+2B| \leq \zj{ \frac{|A+B|}{|A| }}^2 |A'| . \]
 As $A'+B+U \subset A'+2B $, we conclude
  \[ \beta(U) \leq \frac{|A'+B+U|}{\sqrt{|A'| |B| }}\leq \sqrt{\frac{|A'|}{|A|}}\frac{|A+B|^2}{|A| |B|} \leq  \frac{|A+B|^2}{|A| |B|}. \]
As $A$ and $B$ are arbitrary, we conclude 
$$\beta(U) \leq \alpha(U)^2.$$

We approach the first inequality in \eqref{aabb} similarly. Let $A$ and $B$ be arbitrary sets containing $U$.
Then by Lemma~\ref{Plu_alt}
$$
\frac{|B+B|}{|B|} \leq  \frac{|A+B|^2}{|A| |B|},
$$
and $\alpha''(U) \leq \alpha(U)^2$ follows. 

We now proceed to the second statement of \eqref{betaalfa}. Let $A$ and $B$ be sets containing $U$ with $|B| \leq |A|$. 
By Lemma~\ref{Plu} we may find an $A'\subset A$ such that
 \[|B+B+U| \leq  |A'+3B| \leq \zj{ \frac{|A+B|}{|A| }}^3 |A| . \]
Dividing both sides by $|B|$ and using $|B| \leq |A|$ gives $\beta''(U) \leq \alpha(U)^3$.

We now proceed to the second statements of \eqref{aabb}. First, $\beta''(U) \leq \beta''(2U)$ follows immediately from the definitions.  Let  $A,B$ be arbitrary with $|B| \leq |A|$. We find, by Lemma~\ref{Plu}, a $B'\subset B$ such that
 \[ | B' + 2(A+U)| \leq \zj{\frac{|A+B+U| }{   |B|}}^2 |B'|, \]
   and hence
    \[ \frac{ |2A+2U|}{|A|} \leq \frac{|A+B+U|^2 }{  |A| |B|} \]
    and so $\beta''(2U) \leq \beta(U)^2$ follows.
\end{proof}

\tous {This is what easily follows from Pl\"unnecke. Improvements welcome.}

       \begin{Prob}
         How tight are these inequalities? For the discrete cube $K_d=\{0,1 \}^d$
we have $\beta(K_d)=2^d$, $2^{d/2} \leq \alpha(K_d) \leq (3/2)^d$, so $\beta \leq \alpha^2$ is pretty tight, the exponent is definitely not lower than $ \log 2 / \log (3/2)$. 
       \end{Prob}

\section{The independence problem}

In the preceding sections we tacitly assumed that the ambient group $G$ is fixed, and the
sets $A,B$ in the definition of the $\alpha$'s and $\beta$'s are taken from this group. Sometimes
we shall consider different groups, and the possibility of dependence arises.

For this section we change the notations to $\alpha(U, G)$, to indicate the ambient group (and similarly for all other parameters). 

\begin{Conj}[The independence hypothesis]\label{ind}
Let $G$ be a group, $G'$ its subgroup, $U\subset G'$ and let $\vartheta$ be any of the functionals
$\alpha, \alpha', \alpha'', \beta, \beta', \beta''$. We have
 \[ \vartheta(U, G) = \vartheta(U, G'). \]  
\end{Conj}

We cannot even answer Conjecture~\ref{ind} even in the following simple special case. Let $G=\setZ^d$, and assume that $U\subset p \cdot \mathbb{Z}^d$.
 Do we get the same values of $\alpha, \alpha', \alpha'',  \beta', \beta''$
if we restrict $A, B$ to be subsets of $p \cdot \mathbb{Z}^d$?

The only case where we can show this in generality is for $\beta$.

\begin{Th}[Independence for $\beta$]\label{th:betaind}
  Let $G$ be a group, $G'$ its subgroup, $U\subset G'$. We have
   \[ \beta(U, G) = \beta(U, G'). \]
\end{Th}
\begin{proof}
  Take $A,B\subset G$ and split them according to cosets of $G'$, say
   \[ A= \bigcup A_i, \ B = \bigcup B_j . \]
   Assume that $A_1$ is the largest of the $A_i$ and similarly for $B$. The sets $A_1+B_j+U$ are disjoint (as $j$ varies), and hence
    \[ |A+B+U| \geq \sum_j |A_1+B_j+U| \geq  \beta(U, G') \sqrt{|A_1|}  \sum_j \sqrt{ |B_j|}. \]
By symmetry of $A$ and $B$, 
\[ |A+B+U| \geq   \beta(U, G') \sqrt{|B_1|}  \sum_i \sqrt{ |A_i|}. \]
Forming the geometric mean of the above two inequalities and using H\" older of the form
 \[ \sum x_i^2 \leq (\max x_i) \sum x_i\]
separately for the numbers $|A_i|$ and $|B_j|$, we obtain the desired result
\end{proof}
  
An important special case  is easily seen.

\begin{St}[Cartesian products with torsion]\label{tfindep}
  Let $G$ be a group, $G=G' \times H$ with $H$ torsion-free, $U\subset G'$ and let $\vartheta$ be any of the functionals
$\alpha, \alpha', \alpha'',  \beta', \beta''$. We have
 \[ \vartheta(U, G) = \vartheta(U, G'). \] 
\end{St}

In particular this implies that by embedding $\setZ^d$ into $\setZ^k$ with $k>d$ these values do not change.

\begin{proof}
  If $G' $ is a torsion group, then all these functionals have value 1. Assume this is not the case,
and fix a $g\in G'$ of infinite order.

  Take $A, B\subset G$. We are going to construct $A', B'\subset G'$ such that
   \[ |A'+B'| = |A+B|, \ |A'+B'+U| = |A+B+U|,  \]
and the restrictions used to define $\alpha, \alpha', \alpha'',  \beta', \beta''$ are preserved.

Let $H'$ be the subgroup of $H$ generated by the elements in the $H$-projection of $A\cup B$. Since $H$ is torsion-free, we have $H' \cong \setZ^d$ for some $d$. Let 
$e_1, \ldots, e_d$ be a system of generators for $H'$. For fixed integers $m_1, \ldots, m_d$ (to be chosen later) define a homomorphism $\varphi: G' \times H' \to G'$ by
 \[ \varphi(x, y_1e_1 + \ldots + y_de_d) = x+ (m_1y_1 + \ldots + m_dy_d) g .\]
Put $A'=\varphi(A)$, $B'= \varphi(B)$. It is clear that for $m_1, \ldots, m_d$ large enough (and dependent on $A, B, U$), $\varphi$ is one-to-one on $A, B, A+B, A+B+U$ and the claim follows.
\end{proof}

\tous{Probably the assumption can be weakened to that the factor-group $G/G'$ is torsion-free. This
may even not be more general. It may be true that if $G$ is finitely generated (which we can safely assume)
and $G/G'$ is torsion-free, then $G'$ is a direct factor.}

\section{Torsion}

The presence of torsion is the source of difficulties. We conjecture it should not matter much.
 
\begin{Conj}
  Let  $G$ be a group, $H$ its torsion subgroup, $G'=G/H$ the factor group, $\varphi: G \to G'$ the natural homomorphism,
 $U\subset G$, $U'=\varphi(U)\subset G'$ and let $\vartheta$ be any of the functionals
 $\alpha, \alpha', \alpha'', \beta', \beta''$. We have
  \[ \vartheta(U') = \vartheta(U) . \]
\end{Conj}
\begin{Rem}
    The case of $\beta$ follows from Statement~\ref{st:homomorphism} below and supermultiplicativity (Theorem~\ref{Th:beta-mult}) as $\beta$ is always at least 1.
\end{Rem}

\begin{St}[Projections and torsion]\label{st:homomorphism}
   Let  $G$ be a group, $H$ its torsion subgroup, $G'=G/H$ the factor group, $\varphi: G \to G'$ the natural homomorphism,
 $U\subset G$, $U'=\varphi(U)\subset G'$ and let $\vartheta$ be any of the functionals
 $\alpha, \alpha', \alpha'', \beta,  \beta', \beta''$. We have
  \[ \vartheta(U') \geq  \vartheta(U) . \]
\end{St}
\begin{proof}
    For concreteness, let us prove Statement~\ref{st:homomorphism} for the case of $\beta$, as for the other functionals the argument is similar. 
    
    For an arbitrary $\epsilon > 0$ there are $A', B' \subset G'$ such that 
    \[
        \frac{|A'+B'+U'|}{|A'|^{1/2}|B'|^{1/2}} \leq \beta(U') + \epsilon.
    \]
    WLOG we may assume $H$ is of finite order. Take $A := \phi^{-1}(A') $  and $B := \phi^{-1}(B')$, so that $|A| = |H||A'|$ and $|B| = |H||B'|$. At the same time clearly 
    \[
    |A+B+U| \leq |A'+B'+U'||H|,
    \]
    so
    \[
        \beta(U) \leq  \frac{|A+B+U|}{|A|^{1/2}|B|^{1/2}} \leq \beta(U') + \epsilon.
    \]
    The claim follows as $\epsilon$ can be taken arbitrarily close to zero.
\end{proof}

\begin{St}[The trivial lower bounds]\label{st:trivial}
  Let  $G$ be a group, $H$ its torsion subgroup, $U\subset G$. If $U$ is contained in a single coset of $H$, then
   \[ \alpha(U)= \alpha'(U)= \alpha''(U)= \beta(U)= \beta'(U)=\beta'(U)= 1 ,\]
   otherwise
    \[ \beta(U) \geq 2, \ \alpha(U) \geq 3/2. \]
\end{St}
\begin{proof}
    The statement is trivial when $U$ is contained in a coset of $H$.  
    
    Otherwise, we may assume WLOG that $U$ contains the union of $\{0\} \oplus U_0$ and $\{1\} \oplus U_1$ with some $U_0, U_1 \subset G/\setZ$.  We also write 
    \[
     A = \bigsqcup^N_{i =1} a_i \oplus A_i 
    \]
    and
    \[
        B = \bigsqcup^M_{j=1} b_j \oplus B_j 
    \]
 with $A_i, B_j \subset G/\setZ$ and some $N, M$, so that the integers $\{a_i\}$ and $\{b_j\}$ are monotone increasing. Then $|A+B+U|$ contains the disjoint union of 
    \begin{align*}
       &(a_1 + b_1)\oplus (A_1 + B_1 + U_0),\\
       &(a_1 + b_1 + 1) \oplus (A_1 + B_1 + U_1), \ldots, (a_N + b_1 + 1) \oplus (A_N + B_1+ U_1),\\  
       &(a_N + b_2 + 1) \oplus (A_N + B_2+ U_1), \ldots, (a_N + b_M + 1) \oplus (A_N + B_M + U_1). 
   \end{align*}
Since in any group 
\[
|A_i + B_j + U_k| \geq \max \{|A_i|, |B_j|\},
\]
we conclude that 
$$
|A+B+U| \geq \sum^N_{i=1}|A_i| + \sum^M_{j=1}|B_j| = |A| + |B| \geq 2|A|^{1/2}|B|^{1/2}.
$$
In a similar way, for an arbitrary $A \supset U$ holds
$$
|A+B| \geq |A| + 
|B| -\min\{|A_1|, |A_N|, |B_1|, |B_M|\} \geq \frac{3}{2}|A|^{1/2}|B|^{1/2},
$$
and hence $\beta(U) \geq 2$  and $\alpha(U) \geq 3/2$.
\end{proof}

\section{Projection and compression}

By \emph{projection} we mean the application of any homomorphism. We think projections never increase
the value of our  $\alpha$'s and $\beta$'s.

\begin{Conj}[Projection conjecture]
  Let  $G$ be a group, $H$ its subgroup, $G'=G/H$ the factor group, $\varphi: G \to G'$ the natural homomorphism,
 $U\subset G$, $U'=\varphi(U)\subset G'$ and let $\vartheta$ be any of the functionals
 $\alpha, \alpha', \alpha'', \beta', \beta''$. We have
  \[ \vartheta(U') \leq \vartheta(U) . \]
\end{Conj}
\begin{Rem}
  For $\beta_p$ the conjecture follows from Theorem~\ref{betafactor} as $\beta_p \geq 1$ always.
\end{Rem}

\begin{Rem}
  Essentially this means the following. Given sets $A,B\subset G$ (subject to certain conditions, depending
  on which of the functionals we consider) we need to find $A', B'\subset G'$ such that
   \[  \frac{|A'+B'|}{\sqrt{|A'| |B'| }} \leq   \frac{|A+B|}{\sqrt{|A| |B| }} \]
for the $\alpha$'s, or 
\[  \frac{|A'+B'+U'|}{|A'|^{1/p} |B'|^{1-1/p} } \leq   \frac{|A+B+U|}{|A|^{1/p} |B|^{1-1/p} } \]
for the $\beta$'s. The natural approach of taking $A'=\varphi(A)$, $B'=\varphi(B)$ may not work even when $G=\setZ^2$, $G'=\setZ$.
\end{Rem}

We establish an important special case.

\begin{Th}[Projection conjecture with no torsion]\label{projineq}
 Let  $G$ be a group, $H$ its subgroup, $G'=G/H$ the factor group, $\varphi: G \to G'$ the natural homomorphism,
 $U\subset G$ , $U'=\varphi(U)\subset G'$ and let $\vartheta$ be any of the functionals
 $\alpha, \alpha', \alpha'', \beta_p,  \beta', \beta''$. If $H$ is torsion-free, then
  \[ \vartheta(U') \leq \vartheta(U) . \]
\end{Th}

\begin{Def}
  Let  $G$ be a group, $H$ its subgroup, $G'=G/H$ the factor group, $\varphi: G \to G'$ the natural homomorphism.
The \emph{compression along $\varphi$} is the mapping $C_\varphi$ of finite subsets of $G$ into finite subsets of
$G' \times \setZ$ defined as follows. Let $A\subset G$ be a finite set. We put
 \[ C_\varphi(A) = \bigcup_{x\in\varphi(A)} (x \times \{ 0, 1, \ldots, |A \cap \varphi^{-1}(x)| -1 \} . \]
\end{Def}

That is, each part of $A$ in a coset of $H$ is replaced by an interval of the same size. 
If $G=\setZ^d$ and $H=\setZ^k$ with $k<d$, then we can naturally represent the compression in $\setZ^d$,
which is the classical usage of this term.

In what follows we will write $\varphi^{-1}_A(x)$ as an alias for $\varphi^{-1}(x) \cap A$. For a given set $A$ and $x \in G'$, such a set is called the \emph{fiber} of $A$ above $x$. One can say that the compression operator ``normalizes'' each fiber of $A$ by replacing it with an initial segment in $\mathbb{Z}$.

Clearly $ | C_\varphi(A)| = |A|$ always.

\begin{St}[Compressions shrink sumsets]\label{compsum}
   Let  $G$ be a group, $H$ its subgroup, $G'=G/H$ the factor group, $\varphi: G \to G'$ the natural homomorphism,
   $A, B\subset G$. If $H$ is torsion-free, then
    \[   C_\varphi(A) +  C_\varphi(B) \subset  C_\varphi(A+B) .\]
\end{St}
\begin{proof}
   The claim is standard and can be adopted from e.g. \cite{Gardner_Gronchi_2001}. 
   
   Let $z \in \varphi(C_\varphi(A) +  C_\varphi(B))$. There are $z_a \in \varphi(A)$ and $z_b \in \varphi(B)$ such that $z = z_a + z_b$. By the Cauchy-Davenport inequality and the definition of the compression,
   \[
    |\varphi^{-1}_{C_\varphi(A) +  C_\varphi(B)}(z)| = |\varphi^{-1}_{A}(z_a)| + |\varphi^{-1}_{B}(z_b)| - 1 \leq 
    |\varphi^{-1}_{A}(z_a) + \varphi^{-1}_{B}(z_b)| \leq |\varphi^{-1}_{C_\varphi(A+B)}(z)|,
   \]
   and the claim follows.
\end{proof}

\begin{Th}[Compressions]\label{compalphabeta}
   Let  $G$ be a group, $H$ its subgroup, $G'=G/H$ the factor group, $\varphi: G \to G'$ the natural homomorphism,
 $U\subset G$, and let $\vartheta$ be any of the functionals
 $\alpha, \alpha', \alpha'', \beta_p,  \beta', \beta''$. If $H$ is torsion-free, then
  \[ \vartheta(C_\varphi(U)) \leq  \vartheta(U) . \]
\end{Th}

\begin{proof}
  Indeed, the previous statement implies that
\[  \frac{|C_\varphi(A)+C_\varphi(B)|}{\sqrt{|C_\varphi(A)| |C_\varphi(B)| }} \leq   \frac{|A+B|}{\sqrt{|A| |B| }} \]
and
\[  \frac{|C_\varphi(A)+C_\varphi(B)+C_\varphi(U)|}{|C_\varphi(A)|^{1/p} |C_\varphi(B)|^{1-1/p} } \leq   \frac{|A+B+U|}{|A|^{1/p} |B|^{1-1/p} }. \]
Also, the restrictions are preserved (if $A\supset U$, then $C_\varphi(A)\supset C_\varphi(U)$; if $|A| = |B|$, then
$ |C_\varphi(A)| = |  C_\varphi(B)|$).
\end{proof}

\begin{proof}[Proof of Theorem \ref{projineq}]
We can naturally embed $G'$ into $G' \times \setZ$ as $G' \times \{ 0\}$. With this embedding we have $U'\subset C_\varphi(U)$, hence
  
   \[ \vartheta(U,G) \geq \vartheta(C_\varphi (U), G' \times \setZ) \geq  \vartheta(U', G' \times \setZ) =\vartheta(U', G') ; \]
in the last step we apply Statement \ref{tfindep}.  
\end{proof}

\section{Direct product}

The behaviour of our quantities under  direct product and a somewhat more general operation (\emph{tensorization}, see Theorem~\ref{betafactor} below) is important for our applications.

\begin{Conj}[Multiplicativity hypothesis]\label{mult}
Let $G = G_1 \times G_2$, $V_1 \subset G_1$, $V_2\subset G_2$, $U = V_1 \times V_2$, and let $\vartheta$ be any of the functionals
$\alpha, \alpha', \alpha'', \beta_p,  \beta', \beta''$. We have
 \[ \vartheta(U) = \vartheta(V_1) \vartheta(V_2) . \]
\end{Conj}

Submultiplicativity is easy.

\begin{St}[Sub-multiplicativity]\label{sub} 
  Let $G = G_1 \times G_2$, $V_1 \subset G_1$, $V_2\subset G_2$, $U = V_1 \times V_2$, and let $\vartheta$ be any of the functionals
$\alpha, \alpha', \alpha'', \beta_p,  \beta', \beta''$. We have
 \[ \vartheta(U) \leq \vartheta(V_1) \vartheta(V_2) . \]
\end{St}

\begin{proof}
We show only for $\vartheta = \alpha$, the rest being similar. Let $A_1 ,B_1$ be arbitrary sets containing $V_1$ and $A_2 , B_2$ be arbitrary sets containing $V_2$. Then $A_1 \times A_2$ and $B_1 \times B_2$ contain $V_1 \times V_2$ and so
$$\alpha(V_1 \times V_2) \leq \frac{|A_1 \times A_2 + B_1 \times B_2|}{(|A_1||B_1||A_2||B_2|)^{1/2}} = \frac{|A_1 + B_1||A_2 + B_2|}{(|A_1||B_1||A_2||B_2|)^{1/2}}.$$ Thus 
$$\alpha(V_1 \times V_2) \leq \alpha(V_1) \alpha(V_2).$$
\end{proof}

The multiplicativity hypothesis, Conjecture~\ref{mult},  would have consequences for the comparison problems of Section \ref{23ineq}.

\begin{St}
  Let $U$ be a subset of a commutative group.

If Conjecture~\ref{mult} holds for $\alpha'$,  then $\alpha(U)=\alpha'(U) \leq \beta'(U)$.

 If Conjecture~\ref{mult} holds for $\alpha''$,  then $\alpha''(U)\leq \beta''(U)$.

If Conjecture~\ref{mult} holds for $\beta'$,  then $\beta(U) = \beta'(U)$.
\end{St}

\begin{proof}
The inequalities
$$\alpha'(U) \leq \beta'(U) , \ \ \ \alpha''(U) \leq \beta''(U),$$
follow from \eqref{alfabeta}, that is 
$$\alpha'(U) \leq 4 \beta'(U) , \ \ \ \alpha''(U) \leq3 \beta''(U),$$
and Conjecture~\ref{mult} with the tensor power trick. We prove only the first of the two inequalities, the second following similarly. Indeed for any $n \geq 1$, first using Conjecture~\ref{mult} for $\alpha'$ and then Statement~\ref{sub} for $\beta'$, we find
$$\alpha'(U)^n = \alpha'(U^n) \leq 4 \beta'(U^n)\leq 4 \beta'(U)^n .$$
Thus $$\alpha'(U) \leq 4^{1/n} \beta'(U),$$ and the result follows from allowing $n\to \infty$. 
We now show that Conjecture~\ref{mult} implies
$$\alpha(U) = \alpha'(U).$$
By Statement~\ref{basicprops}, it is enough to show $\alpha'(U) \leq \alpha(U)$. Let $A$ and $B$ be sets containing $U$. Then $A\times B$ and $B\times A$ contain $U\times U$ and are of the same size, so
$$\alpha'(U^2) \leq \frac{|A\times B + B \times A|}{|A||B|} = \frac{|A+B|^2}{|A||B|}.$$
The result then follows from Conjecture~\ref{mult} for $\alpha'$. The inequality $\beta(U) = \beta'(U)$ is similar. 
\end{proof}

We are far from knowing the multipliciativity of $\alpha$, as we cannot even compute $\alpha(\{0,1\}^d)$. We do know multiplicativity of $\beta$.

\begin{Th}[Multiplicativity of $\beta$]\label{Th:beta-mult}
  Let $G = G_1 \times G_2$, $V_1 \subset G_1$, $V_2\subset G_2$, $U = V_1 \times V_2$. We have
   \[ \beta_p(U) =  \beta_p(V_1) \beta_p(V_2). \]
\end{Th}

This will follow from supermultiplicativity, which we shall establish in a more general setting.

\begin{Th}[$\beta$ along fibers]\label{betafactor}
   Let  $G$ be a group, $H$ its subgroup, $G'=G/H$ the factor group, $\varphi: G \to G'$ the natural homomorphism,
   $U\subset G$, $V=\varphi(U)$. We have
    \[  \beta_p(U) \geq  \beta_p(V) \min_{x\in V} \beta_p\bigl( U \cap \varphi^{-1}(x) \bigr) . \]
\end{Th}

If $H$ is a direct factor, this can be reformulated as follows.

\begin{Cor} \label{thm:beta_direct_factor}
  Let $G = G_1 \times G_2$, $V \subset G_1$, and for each $x\in V$ given a set $W_x\subset G_2$. Put
   \[ U= \bigcup_{x\in V} \{ x\} \times W_x .\]
We have
   \[ \beta_p(U) \geq  \beta_p(V) \min_{x\in V} \beta_p(W_x) . \]
\end{Cor}
\tous{An alternative form is as follows: 
\begin{Th}
  Let $G = G_1 \times G_2$,  $U\subset G$. For $x\in G_1$, define $V_x = \{y\in G_2: (x,y)\in U \}$. Put $U_1= \{x\in G_1: V_x \neq \emptyset \}$. We have
   \[ \gamma(U) \geq \gamma(U_1) \min_{x\in U_1} \gamma(V_x) . \]
In particular, if  $U = V_1 \times V_2$, then $\gamma(U)\geq \gamma(V_1) \gamma(V_2)$.
\end{Th}
Which is better?  \\
Dmitrii: $\gamma$ is not yet defined here, so I'd stick to $\beta$ here}

Theorem~\ref{betafactor} (and thus Theorem~\ref{Th:beta-mult} and Corollary~\ref{thm:beta_direct_factor}) will be proved in a yet more general form in Section~\ref{funcdir}. It turns out that a functional analog of $\beta$ that we introduce shortly, provides greater flexibility for carrying out an induction argument.

\vskip 1cm
\centerline{ \large{ \bf Part II: functions}}
\vskip 1cm
    
\section{Functional tripling}

We shall consider nonnegative-valued functions in the space $\ell^1(G)$. A set $A$  naturally
corresponds to the function $\1A$.

\begin{Def}
  The \emph{max-convolution} of the functions $f,g$ is
   \[ (f\C g)(x) = \max_t  f(t)g(x-t) .\]
\end{Def}

This generalizes the notion of sumset. For the indicator functions $\1A, \1B$ of sets $A, B$ we have
 \[ \1A \C \1B = \1{A+B} .\]

One can replace the notion of cardinality of a set with the $\ell^1$ norm of a function. However, we have a more robust notion. 

\begin{Def}
  The \emph{level sets} of a function $f$ are the sets
   \[ \f(t) = \{ x\in G: f(x) \geq t \} . \]
   The \emph{distribution function} of $f$ is the function $F: \setR^+ \to \setZ$ given by
    \[ F(t) = | \f(t) | .\]
\end{Def}

Note that this is different from the definition used in probability theory.

\begin{Def}
  Let $f,g$ be functions with distribution functions $F, G$. If $F=G$, we call them
\emph{identically distributed} and write $f \sim g$.
\end{Def}

 \begin{Def}\label{gamma} The \emph{functional triplings} of a function $f$ are the quantities
    \[ \gamma(f) = \inf_{g,h} \frac{ \norm{f\C g\C h}_1}{ \norm{g}_2  \norm{h}_2 } ,\]         
the \emph{unrestricted tripling};
    \[ \gamma_p(f) = \inf_{g,h} \frac{ \norm{f\C g\C h}_1}{ \norm{g}_p  \norm{h}_q } ,\]    
its asymmetric variant, where $1/p+1/q=1$;     
\[ \gamma'(f) = \inf_{g \sim h} \frac{ \norm{f\C g\C h}_1}{ \norm{g}_2  \norm{h}_2 } ,\]       
the \emph{isometric tripling};
 \[ \gamma''(f) = \inf_{g} \frac{ \norm{f\C g\C g}_1}{ \norm{g}_2^2  } ,\]         
the \emph{isomeric tripling}.
       \end{Def}

       \begin{Conj}
          \[ \gamma = \gamma' = \gamma'' .\]
       \end{Conj}

\tous{Can we bound $\gamma''$ in terms of $\gamma$? Since we don't have a ``functional Pl\"unnecke'',
we cannot apply the same approach as for $\beta$.}

Tripling of sets can be expressed via functional tripling.

\begin{Th}[Function and Set analog of $\beta$ are the same]\label{th:betaisgamma}
  Let $U$ be any finite set in a commutative group. We have
   \[ \beta_p(U)=\gamma_p(\1U), \ \beta'(U)=\gamma'(\1U), \ \beta'(U)=\gamma''(\1U). \]
\end{Th}
\begin{proof}

We prove only $\beta_p(U) = \gamma_p(1_U)$, as the other equalities follow similarly (in fact the definitions of $\gamma' , \gamma''$ are designed just for this). We have
$$\frac{|A+B+U|}{|A|^{p} |B|^{q}} = \frac{ \norm{1_U\C 1_A\C 1_B}_1}{ \norm{1_A}_p  \norm{1_B}_q },$$
and the inequality $\gamma_p(1_U) \leq \beta_p(U)$ follows from taking an infimum over $A$ and $B$. 

To prove the reverse inequality, we need a lemma, which is a multiplicative analog of Pr\'ekopa-Leindler, Theorem~\ref{PL}.

\begin{Lemma}[Multiplicative Pr\'ekopa-Leindler] \label{lm:prekopa-leidner}
  Let $F, G, H$ be measurable functions $\mathbb{R}_+ \to [0, 1]$ and $1 < p, q < \infty $ are H\"older conjugates, that is 
  $$
  \frac{1}{p} + \frac{1}{q} = 1.
  $$
  Assume that for any $u, v \in \mathbb{R}_{+}$
  $$
  H(uv) \geq F(u)G(v).
  $$
  Then 
  $$
  \|H\|_1 \geq \left(\int^1_{0} F^p(t^{1/p}) dt\right)^{1/p} \left(\int^1_{0} G^q(t^{1/q}) dt\right)^{1/q}
  $$
\end{Lemma}
\begin{proof}
  Define 
  $$
  h(x) := H(e^{-x})e^{-x}
  $$
  and further
  \begin{eqnarray}
  f(x) &:=& F^p(e^{-x/p})e^{-x} \\
  g(x) &:=& G^q(e^{-x/q})e^{-x}.
  \end{eqnarray}
 We then have that for any $x, y > 0$
 $$
 h\left(x/p + y/q\right) \geq f^{1/p}(x) g^{1/q}(y),
 $$
 so by the Pr\'ekopa-Leindler  inequality, Theorem~\ref{PL},
 $$
 \| h \|_1 \geq \|f\|^{1/p}_1 \|g\|^{1/q}_1
 $$
 The claim follows after the change of variables $t = e^{-x}$.
\end{proof}

   We want to prove that for any non-negative functions $g,h$
   $$
   \| 1_V \C g \C h \|_1 \geq \beta (V) \|g\|_p \|h\|_q.
   $$
After rescaling, we may assume $\max g = \max h = 1$.  Let 
$$
S(t) := \{z : 1_V \C g\C h(z) \geq t\}.
$$
Further define the distribution functions
$$
G(t) := \{z : g(z) \geq t\}
$$
and 
$$
H(t) := \{z : h(z) \geq t\}.
$$

For any $u, v$ then have the inclusion
$$
  G(u) + H(v) + V \subset S(uv),
$$
so
$$
|S(uv)| \geq \beta_p(V) |G(u)|^{1/p} |H(v)|^{1/q}.
$$
It follows from Lemma~\ref{lm:prekopa-leidner} that 
$$
\| 1_V \C g \C h \|_1 \geq \beta_p(V) \left(\int^1_{0} G(t^{1/p}) dt\right)^{1/p} \left(\int^1_{0} H(t^{1/q}) dt\right)^{1/q}.
$$
The result now follows from the layer-cake principle of the form
$$\int F(t^{1/p})dt = \int f(t)^pdt.$$

\end{proof}

\section{The independence problem}

The independence problem arises as it did for sets.

For this section we change the notations to $\gamma(U, G)$ etc. to indicate the ambient group.

\begin{Th}[Ambient group does not change $\gamma$]
  Let $G$ be a group, $G'$ its subgroup, $f\in\ell^1( G')$. We have
 \[ \gamma_p(f, G) = \gamma_p(f, G'). \]  
\end{Th}

\begin{proof}
This follows from Theorem~\ref{th:betaisgamma} and Theorem~\ref{th:betaind}. 
\end{proof}
\begin{Conj}[Functional independence hypothesis]
Let $G$ be a group, $G'$ its subgroup, $f\in\ell^1( G')$ and let $\vartheta$ be any of the functionals
$\gamma', \gamma''$. We have
 \[ \vartheta(f, G) = \vartheta(f, G'). \]  
\end{Conj}

\section{Direct product}\label{funcdir}

\begin{Th}[Multiplicativity]
   Let $G = G_1 \times G_2$,  $f_1\in\ell^1( G_1)$,  $f_2\in\ell^1( G_2)$,  and define $f\in\ell^1( G)$ by 
   $f(x,y)=f_1(x)f_2(y)$. We have
    \[ \gamma_p(f) =  \gamma_p(f_1) \gamma_p(f_2), \]
\[ \gamma'(f) \leq \gamma'(f_1) \gamma'(f_2), \]
\[ \gamma''(f) \leq \gamma''(f_1) \gamma''(f_2). \] 
\end{Th}

\begin{proof}
The $\leq$ inequalities all follow from the fact that (with $g_i, h_i$ defined similarly to $f_i$) 
$$
f \C g \C h(x, y) \leq f_1 \C g_1 \C h_1(x) f_2 \C g_2 \C h_2(y)
$$
for any $x \in G_1$ and $y \in G_2$. 

The reverse inequality for $\gamma_p$ follows from a much more general Theorem~\ref{gammafactor} (and Theorem~\ref{thm:cart_tenzor_gamma}) towards which we immediately proceed.

\end{proof}

\begin{Th}[Tensorization]\label{gammafactor}
   Let  $G$ be a group, $H$ a subgroup, $G'=G/H$ the factor group, $\varphi: G \to G'$ the natural homomorphism,
   $f\in\ell^1( G)$. Define $f_{\varphi} \in\ell^1( G')$ by
    \[ f_{\varphi} (x) := \gamma_p \bigl(f|_{ \varphi^{-1}(x)} \bigr).    \]
     We have $\gamma_p(f)  \geq \gamma_p(f_{\varphi})$.
\end{Th}
\begin{proof}


Let $g,h \in \ell^`(G)$ be non-negative. We have
\begin{align}
||f \C g \C h||_1 &= \sum_{z\in G} \max_{x_1 + x_2 + x_3 = z} f(x_1) g(x_2) h(x_3) \nonumber \\
& =\sum_{z_1 \in G'}  \sum_{z \in z_1 + H} \max_{\substack{w_1 , w_2 , w_3  \in G' \\ w_1 + w_2 + w_3 = z_1}} \max_{\substack{x_1 + x_2 +x_3 = z \\ x_i \in H + w_i}} f(x_1) g(x_2) h(x_3) \nonumber \\
& \geq \sum_{z_1 \in G'} \max_{\substack{w_1 , w_2 , w_3  \in G' \\ w_1 + w_2 + w_3 = z_1}} \sum_{z \in z_1 + H} \max_{\substack{x_1 + x_2 +x_3 = z \\ x_i \in H + w_i}} f(x_1) g(x_2) h(x_3) \label{eq:last}  
\end{align}
Define the functions $g', h'$ on $G'$ as follows
$$
g'(w_2) := ||g|_{\varphi^{-1}(w_2)}||_p, h'(w_3) := ||h|_{\varphi^{-1}(w_3)}||_q.
$$
One can further estimate (\ref{eq:last})
\begin{align*}
& \geq \sum_{z_1 \in G'} \max_{\substack{w_1 , w_2 , w_3  \in G' \\ w_1 + w_2 + w_3 = z_1}} f_{\varphi} (w_1) g'(w_2) h'(w_3) \\
& \geq \gamma_p(f_{\varphi}) ||g||_p ||h||_q
\end{align*}
The last inequality follows from the observation that 
$$
||g'||_p = ||g||_p, ||h'||_q = ||h||_q.
$$
\end{proof}

We can specialize Theorem~\ref{gammafactor} to cartesian products. 

\begin{Th} \label{thm:cart_tenzor_gamma}
  Let $G = G_1 \times G_2$, $f$ a function on $G$. For $x\in G_1$, define $f_x(y)=f(x,y)$, functions on $G_2$.
Let $g(x) = \gamma(f_x)$, a function on $G_1$. We have $\gamma_p(f)  \geq \gamma_p(g)$.
\end{Th}
\begin{proof}
We let $\varphi: G_1 \times G_2 \to G_1$ by projection. 
Then 
$$f_{\varphi}(x) = \gamma(f_x),$$
and so the result follows from Theorem~\ref{gammafactor}. 
\end{proof}

We are now ready to prove Theorem~\ref{betafactor}. 

\begin{proof}[Theorem~\ref{betafactor}]
Let $U$ be as in the statement and $f = 1_U = :  U $. Then, by Theorem~\ref{th:betaisgamma}, for $x\in V$ 
$$
 f_{\varphi}(x) := \gamma_p\left(f|_{\varphi^{-1}(x)} \right) = \beta_p(U \cap \varphi^{-1}(x)).
$$
In particular, we have the point-wise bound
$$
f_{\varphi}(x) \geq 1_V(x) \min_{t \in V}\beta_p(U \cap \varphi^{-1}(t)),
$$
so again by Theorem~\ref{th:betaisgamma} and linearity
$$
\gamma_p(f_{\varphi}) \geq \beta_p(V) \min_{t \in V}\beta_p(U \cap \varphi^{-1}(t)).
$$
The result now follows from Theorem~\ref{gammafactor}, as 
$$
\beta_p(U) = \gamma_p(f) \geq \gamma_p(f_{\varphi}) \geq \beta_p(V) \min_{x \in V}\beta_p(U \cap \varphi^{-1}(x)).
$$
\end{proof}

\section{Functional tripling for functions supported on quasicubes}

The goal of this section is to prove Theorem~\ref{thm:quasicube_beta}. The basic strategy is to use tensorization to reduce to a two-point inequality. We let $U$ be a $d$ dimensional quasi-cube, defined in Defintion~\ref{Def:quasicube}. Thus $U$ is $d$ dimensional as thus we may assume 
$$U \subset H' \times \mathbb{Z}^d,$$
where $H'$ is a torsion subgroup. We first show we may assume $H' = \{0\}$. Let $\pi$ be the projection to $\mathbb{Z}^d$. Then by induction, it follows that $|\pi(U)| = |U|$. Thus for $V \subset U$, to prove Theorem~\ref{thm:quasicube_beta}, 
$$\beta(\pi(V)) \geq |V|.$$
By Theorem~\ref{th:betaind}, we may assume 
$$V \subset U \subset \mathbb{Z}^d.$$
Central to our study will be the following.
\begin{Th}[$\gamma$ for two-point functions]\label{Th:gamma2point}
For $0 \leq \delta \leq 1$, 
$$f_{\delta} := 1_{\{0\}} + \delta 1_{\{1\}}.$$
Then 
$$\gamma(f_{\delta}) = \delta +1,$$
and more generally,
$$\gamma_p(f_{\delta})  \geq \frac{p^{1/p} q^{1/q}}{2}(1 + \delta).$$
\end{Th}
We remark that the stronger
$$\gamma_p(f_{\delta}) \geq (\delta^{c} +1)^{1/c} ,  \ \ \ 2^{1/c} = p^{1/p} q^{1/q},$$
is probably true, though we do not require it (see this mathoverflow post \cite{MOpost}). Such a result would be useful for quasi-cubes that are asymmetrical in size. We also remark that Pr\'ekopa \cite[Equation 2.4]{Prekopa} proved Theorem~\ref{Th:gamma2point} in the special case $p=q=2$ and $\delta =1$ and his proof extends to the $\delta = 0$ case. 

We now present an important family of examples. First, they are a natural guess for minimizers of $\gamma
_p(f_{\delta})$ and indeed show that Theorem~\ref{Th:gamma2point} is best possible. Secondly, in the proof of Theorem~\ref{Th:gamma2point} below, we show that to bound $\gamma_p(f_{\delta})$ from below it is enough to consider $g$ and $h$, from  Definition~\ref{gamma}, from the following. 
\begin{Ex}\label{geom}
Fix $0 < \delta < 1$. Let $g = (1 , \delta , \ldots , \delta^r)$ and $h = (1 , \delta , \ldots , \delta^s)$. Then 
$$||f_{\delta} \C g \C h ||_1 = \sum_{j=0}^{r+s+1} \delta^j = \frac{1 - \delta^{r+s+2}}{1-\delta},$$
while
$$||g||_p^p = \frac{1 - \delta^{p(r+1)}}{1-\delta^p},$$
and 
$$||h||_q^q = \frac{1-\delta^{q(s+1)}}{1-\delta^q}.$$
Note that 
\begin{equation} \label{same} \frac{1-\delta^{r+s+2}}{(1-\delta^{p(r+1)})^{1/p}(1-\delta^{q(s+1)})^{1/q}} \geq 1.\end{equation}
Indeed, this follows from the inequality
$$(1-x y) \geq (1-x^p)^{1/p} (1-y^q)^{1/q} , \ \ \ 0 \leq x,y \leq 1,$$
which is an application of H\"older's inequality applied to the vectors
$$(x^n)_{n \in \mathbb{Z}_{\geq 0}} , \ \  \ (y^n)_{n \in \mathbb{Z}_{\geq 0}} .$$
Thus \eqref{same} is minimized by allowing $r,s \to \infty$ and so 
$$\frac{||f_{\delta} \C g \C h||_1}{||g||_p ||h||_q} \geq \frac{(1-\delta^p)^{1/p} (1-\delta^q)^{1/q}}{1-\delta}.$$
In the most important case, that is $p=q=2$, we have 
$$\frac{||f_{\delta} \C g \C h||_1}{||g||_2 ||h||_2} \geq 1+ \delta,$$ 
while the more general case is a bit harder.
\begin{Lemma}\label{minim}[Minimizer for $\gamma_p$]
Let $0 < \delta <1$ and 
$$g = h = (1, \delta ,\delta^2 \ldots, ).$$
Then we have $$\frac{||f_{\delta} \C g \C h||_1}{||g||_p ||h||_q} \geq \frac{p^{1/p} q^{1/q}}{2} (1+\delta).$$
\end{Lemma}
\begin{proof}
Similar to the $p=q=2$ case above it suffices to show that
$$\frac{(1-\delta^p)^{1/p} (1-\delta^q)^{1/q}}{1-\delta} \geq \frac{p^{1/p} q^{1/q}}{2} (1+\delta)$$
or
$$
(1-\delta^{p})^{1/p}(1-\delta^{q})^{1/q} \geq \frac{p^{1/p} q^{1/q}}{2}(1-\delta^2).
$$
This follows upon applying H\"older's ienquality:
$$\int_{\delta}^1 s^2 \frac{ds}{s} \leq \left(\int_{\delta}^1 s^p \frac{ds}{s} \right)^{1/p} \left(\int_{\delta}^1 s^q \frac{ds}{s} \right)^{1/q} .$$
\end{proof}
\end{Ex}

\begin{proof}[Proof of Theorem~\ref{thm:quasicube_beta} and Theorem~\ref{mainthm}.] 
We first show how Theorem~\ref{Th:gamma2point} implies Theorem~\ref{thm:quasicube_beta}. Thus we assume $p=q=2$. 
By the discussion at the beginning at the current section we may assume that $V \subset \mathbb{Z}^d$. By Theorem~\ref{th:betaind} we can further assume that in fact $V \subset \mathbb{Q}^d$. Since $\beta$ is invariant under bijective linear transformations of $\mathbb{Q}^d$, after a suitable translation and choosing a basis $\{e_i\}$ for the ambient group $\mathbb{Q}^d$ (now viewed as a linear space) one can WLOG write 
$$V = 0 \oplus V_0 \cup e_1 \oplus V_1,$$
where $V_0$ and $V_1$ are $d-1$ dimensional quasi-cubes. 

Now, we again use the independence of the ambient group (Theorem~\ref{th:betaind}) to reduce the ambient group to the one generated by $V$, so that now $V \subset G := \mathbb{Z}e_1 \times \mathbb{Z}^{d-1}$.

Let $A,B \subset \mathbb{Z}^d$ and write
$$A = \bigcup_i ie_1 \oplus A_i , \ \ \ B =\bigcup_j je_1 \oplus B_j.$$

We claim that $\gamma(\1V) = |V|$ and in particular
\begin{equation}\label{inter} 
|A+B+V| \geq ||f \C g \C h||_1, 
\end{equation}
where 
$$f = |V_0| 1_0 + |V_1| 1_1, \ \ \ g(i) = |A_i|^{1/2} , \ \ \ h(j) = |B_j|^{1/2}.$$

We induct on the dimension on $V$. The base case follows directly from Theorem~\ref{Th:gamma2point} and linearity of $\gamma$.

We let $\pi_1$ be projection onto the first coordinate and 
$$
X_k = \pi^{-1}(k) \cap (A+B+V).
$$
Then, $X_k$ contains all the fiber sumsets as long as the first coordinate equals $k$, so
\begin{align*}|X_k| &= \max \{ \max_{i+ j = k} |A_i + B_j + V_0| , \max_{i+j = k-1} |A_i + B_j + V_1|\} \\ 
& \geq \max\{\max_{i+ j = k} \beta(V_0) |A_i|^{1/2} |B_j|^{1/2} , \max_{i+j = k-1}\beta(V_1) |A_i|^{1/2} |B_j|^{1/2}\} \\
& = \max\{\max_{i+ j = k} \gamma(\1{V_0}) |A_i|^{1/2} |B_j|^{1/2} , \max_{i+j = k-1}\gamma(\1{V_1}) |A_i|^{1/2} |B_j|^{1/2}\} \\
& = \max\{\max_{i+ j = k} |V_0| |A_i|^{1/2} |B_j|^{1/2} , \max_{i+j = k-1}|V_1| |A_i|^{1/2} |B_j|^{1/2}\} \\
& = f \C g \C h (k).\end{align*}

Summing over $k$ gives \eqref{inter}.  By Theorem~\ref{Th:gamma2point}, we have
$$|A+B+V| \geq ||f \C g \C h||_1 \geq (|V_1| + |V_2|) ||f||_2 ||g||_2 = |V| |A|^{1/2} |B|^{1/2},$$
which implies Theorem~\ref{thm:quasicube_beta}.

We now handle the case of general $p$. Everything proceeds the same as in the $p=2$ case, except the induction claim is
$$\gamma_p(\1V) \geq \frac{|V|p^{d/p} q^{d/q}}{2^d}.$$

Theorem~\ref{mainthm} now follows from Theorem~\ref{th:betaisgamma}.

\end{proof}

We now proceed to the proof of Theorem~\ref{Th:gamma2point}. We first need the following lemma.
\begin{Lemma}[$\gamma$ and permutations]\label{increasing}
Let $f,g,h : \mathbb{Z} \to \mathbb{R}$ be non-negative functions with finite support. Let $\sigma ,\tau , \rho$ be permuations of the support of $f,g,h$, respectively. Set
$$
f_{\sigma}(x) : = f\circ \sigma(x),
$$
and similarly for $g$ and $h$. Then 
$$
||f_{\sigma} \C g_{\tau} \C h_{\rho}||_1,
$$ is minimized for a choice of permutations that makes each function non-increasing.
\end{Lemma}

\begin{proof}
We may suppose, after translation, that the smallest element of the support is zero for all three functions. Put $F:=f \C g \C h$ and let $\sigma, \tau, \rho$ be some permutations such that $f_\sigma,g_\tau, h_\rho$ are non-increasing. Note that $G := f_{\sigma} \C g_{\tau} \C h_{\rho}$ is then also non-increasing. Let $s$ be a sufficiently large number and order the sequence $F(0), \ldots, F(s)$ (which will end with zeroes) via 
\begin{equation} \label{eq:ordered_support}
F_0 \geq \cdots \geq F_s.
\end{equation}
We claim that for $0 \leq v \leq s$
$$
f_\sigma \C g_\tau \C h_\rho(v) = G(v) \leq F_v,
$$ 
and the result follows from this claim. Let $m, n, r$ be such that
$m+n+r = v$ and 
$$
f_{\sigma}(m)g_\tau(n)h_\rho(r) = G(v).
$$
It follows by the choice of $\sigma, \tau, \rho$, that for any $i\leq m$, $j \leq n$, $k \leq r$
\begin{eqnarray*} 
G(v) &=& f_{\sigma}(m)g_\tau(n)h_\rho(r) \\
&\leq& f_{\sigma}(i)g_\tau(j)h_\rho(k) \\
&=& f(\sigma(i))g(\tau(j))h(\rho(k)) \\
&\leq& F(\sigma(i) + \tau(j) + \rho(k)) 
\end{eqnarray*}
It follows that for any 
\begin{equation} \label{eq:perm_sets}
 t \in \{\sigma(0), \ldots, \sigma(m)\} + \{\tau(0), \ldots, \tau(n)\} + \{\rho(0), \ldots, \rho(r)\}   
\end{equation}
holds
$$
G(v) \leq F(t).
$$
But the sumset on the RHS of (\ref{eq:perm_sets}) is of size at least $m+n+r+1$, by the Cauchy-Davenport inequality. Thus, there are at least $v+1$ values in the sequence (\ref{eq:ordered_support}) that are no less than $G(v)$, and hence
$$
G(v) \leq F_v.
$$
\end{proof}

\begin{proof}[Proof of Theorem~\ref{Th:gamma2point}]

We aim to show that for non-negative valued $g$ and $h$ in $\ell_1(\mathbb{Z})$
\begin{equation} \label{ratio} \frac{||f_{\delta} \C g \C h||}{||g||_p ||h||_q}\ge c_{\delta}=\frac{(1-\delta^p)^{1/p} (1-\delta^q)^{1/q}}{1-\delta},\end{equation}

We remind the reader that $c_{\delta}$ is the infimum of \eqref{ratio} over all $g,h$ of the form \begin{equation}\label{form} g = (1 , \delta , \ldots , \delta^r) , \ \ \ h = (1 , \delta , \ldots , \delta^s),\end{equation}

as shown in Example ~\ref{geom}.

We suppose there is a $g,h \in \ell^{1}(\mathbb{Z})$ such that \eqref{ratio} is smaller than $c_{\delta}$. By continuity, we may suppose both $g$ and $h$ have finite support. By Lemma~\ref{increasing}, we may permute $g,h$ so they are both non-increasing.  After translation of the supports, we suppose 
$$g = (g_0 , \ldots g_r), \ \ \ h = (h_0 , \ldots , h_s).$$
We further assume that $r+s$ is minimally chosen. Thus 
$$\frac{||f_{\delta} \C u \C v||_1}{||u||_p ||v||_q} \geq c_{\delta} ,$$
for any $u$ and $v$ satisfying
$$|{\rm supp}(u)| + |{\rm supp}(v)| < r+s+2.$$
By compactness, there exists $g,h$ which minimize \eqref{ratio}
subject to ${\rm supp}(g) \subseteq \{0, ..., r\}$, ${\rm supp}(h) \subseteq \{0, ..., s\}$, and 
\begin{equation}\label{norm} ||g||_p = 1 ,  \ \ \ ||h||_q = 1 .\end{equation}

Because the value of \eqref{ratio} doesn't change by multiplying $g$ and $h$ by constant, this is also a minimum over all $g$ and $h$ where ${\rm supp}(g) \subseteq \{0, ..., r\}$ and ${\rm supp}(h) \subseteq \{0, ..., s\}$

Set
$$p(x) = f_{\delta} \C g \C h(x) .$$
Let 
$$Q_1 \subset \{0 , \ldots , r\}, \ \ \ Q_2 \subset \{0 , \ldots , s\}.$$

We let $R(Q_1)$ be the set of all indices $n \in \{0 , \ldots , r+s+1\}$ such that if 
$$p_n = f_i g_j h_k , \ \ \ (i+j + k = n),$$
then $j \in Q_1$, and similarly for $Q_2$. We now analyze what happens when we replace $g$ with $g_t$ which we get by multiplying all $g(i)$ by $(1-t)$ where $i\in Q_1$. ($t$ is a small enough positive real number.)

In $f_{\delta}\C g_t \C h$, the values corresponding to $R(Q_1)$ will be multiplied by $(1-t)$, the other values will be the same as in $f_{\delta}\C g \C h$.

$$r(t) : =  \frac{||f_{\delta} \C g_t \C h||_1}{||g_t||_p ||h||_q}.$$

$r'(0)$ is the right-hand derivative of $r(t)$ at 0. By minimality, $r'(0)\ge 0$.

The right-hand derivative of $||f\C g\C h||_1$ at $0$ is $$-\sum_{y \in R(Q_1)} p(y)$$

and the right-hand derivative of $||g_t||_p$ is $$-\frac{1}{p}\frac{1}{||g_t||_p^{p+1}}p(1-t)^{p-1}(-1)\sum_{x \in Q_1} g(x)^p$$ which is equal to $$\frac{\sum_{x \in Q_1} g(x)^p}{ ||g||_p^{p+1} ||h||_q}$$ at $0$.

So $$ 0 \leq r'(0) = \frac{ ||p||_1 \sum_{x \in Q_1} g(x)^p}{ ||g||_p^{p+1} ||h||_q} - \frac{\sum_{y \in R(Q_1)} p(y)}{||g||_p ||h||_q}$$

By symmetry we get a similar inequality for any $Q_2 \subset \{0 , \ldots , s\}$, so by reframing the inequalities

\begin{equation}\label{min} ||p||_1^{1/p} \geq \frac{||g||_p\left( \sum_{y \in R(Q_1)} p(y) \right)^{1/p}}{ \left(\sum_{x \in Q_1} g(x)^p \right)^{1/p} }, \ \ \  ||p||_1^{1/q} \geq \frac{||h||_q\left( \sum_{y \in R(Q_2)} p(y) \right)^{1/q}}{ \left(\sum_{x \in Q_2} h(x)^p \right)^{1/p} } \end{equation} 
We now define new functions,
$$a(i) = \delta^{-i} g(i) , \ \ \ b(j) = \delta^{-j} h(j).$$
We set $$Q_1 = \{i : a(i) = \max a\} , \ \ \ Q_2 = \{j : b(j) = \max b\}.$$
and set 
$$u(i) = g(i)1_{Q_1}(i) , \ \ \ v(j) = h(j) 1_{Q_2}(j).$$

$$R(Q_1)\supseteq {\rm supp}(f_{\delta}\C u\C v),$$ 
$$R(Q_2)\supseteq {\rm supp}(f_{\delta}\C u\C v)$$

So $\sum_{y \in R(Q)} p(y)\ge ||f_{\delta}\C u\C v||_1$ is true for both $Q_1$ and $Q_2$.

Combining it with \eqref{min}, we get 

\begin{equation}\label{sides} ||p||_1^{1/p} \geq \frac{||g||_p ||f_{\delta}\C u\C v||_1^{1/p}}{ ||u||_p }, \ \ \  ||p||_1^{1/q} \geq \frac{||h||_q||f_{\delta}\C u\C v||_1^{1/q}}{||v||_q } \end{equation} 

Multiplying the two inequalities in \eqref{sides}, we get 

$$\frac{||f_{\delta} \C g \C h||}{||g||_p ||h||_q}\ge \frac{||f_{\delta} \C u \C v||}{||u||_p ||v||_q}$$

If either $a$ or $b$ is not constant, then $u$ or $v$ has a value of $0$ at some point. Then by Lemma~\ref{increasing} we can rearrange it to a non-increasing order, with making $$\frac{||f_{\delta} \C u \C v||}{||u||_p ||v||_q}$$ smaller or equal after the rearrangement.

Now $|{\rm supp}(u)|+|{\rm supp}(v)|<|{\rm supp}(g)|+|{\rm supp}(h)|$, so because we started with a counterexample with minimal supports, $$\frac{||f_{\delta} \C u \C v||}{||u||_p ||v||_q}\ge c_{\delta}$$

This is a contradiction, because we assumed that $$\frac{||f_{\delta} \C g \C h||}{||g||_p ||h||_q}<c_{\delta}$$

So $a$ and $b$ must both be constant, but then \eqref{ratio} can't be smaller than $c_{\delta}$, as proved in Example ~\ref{geom}..
Thus the ratio in \eqref{ratio} is at least $c_{\delta}$. By Lemma ~\ref{minim}
$$c_{\delta} \geq p^{1/p}q^{1/q}\frac{(1+\delta)}{2}.$$
\end{proof}

\section*{Acknowledgement}
GS is supported by Ben Green's Simons Investigator Grant 376201. DZ is supported by a Knut and Alice Wallenberg Fellowship (Program for Mathematics 2017). IR is supported by Hungarian National Foundation for Scientific Research (OTKA), Grants No. T 29759 and T 38396. The authors thank Mat\'e Matolsci, Thomas Bloom, Misha Rudnev, Oliver Roche-Newton, Ben Green and Ilya Shkredov for useful discussions.

\bibliographystyle{alpha}
\bibliography{main}

\newcommand{\etalchar}[1]{$^{#1}$}
\begin{thebibliography}{BDF{\etalchar{+}}11}

\bibitem[BC04]{BC}
Jean Bourgain and Mei-Chu Chang.
\newblock {On the size of $k$-fold sum and product sets of integers}.
\newblock {\em Journal of the American Mathematical Society}, 17(2):473--497,
  2004.

\bibitem[BDF{\etalchar{+}}11]{Bour}
Jean Bourgain, Stephen Dilworth, Kevin Ford, Sergei Konyagin, Denka Kutzarova,
  et~al.
\newblock Explicit constructions of {RIP} matrices and related problems.
\newblock {\em Duke Mathematical Journal}, 159(1):145--185, 2011.

\bibitem[Bec75]{Beck}
William Beckner.
\newblock {\em Inequalities in Fourier analysis.}
\newblock PhD thesis, Princeton., 1975.

\bibitem[Gar02]{Gardner}
Richard Gardner.
\newblock {The Brunn--Minkowski inequality}.
\newblock {\em Bulletin of the American Mathematical Society}, 39(3):355--405,
  2002.

\bibitem[GG01]{Gardner_Gronchi_2001}
R.~J. Gardner and P.~Gronchi.
\newblock {A Brunn--Minkowski inequality for the integer lattice}.
\newblock {\em Transactions of the American Mathematical Society},
  353(10):3995–4025, Oct 2001.

\bibitem[GMR{\etalchar{+}}]{GMRSZ}
B.~Green, D.~Matlosci, I.~Ruzsa, G.~Shakan, and D.~Zhelezov.
\newblock {A weighted Pr\'ekopa-Leindler Inequality and sumsets with
  quasicubes}.
\newblock preprint.

\bibitem[GT06]{GreenTao}
B.~{Green} and T.~{Tao}.
\newblock Compressions, convex geometry and the {F}reiman{--}{B}ilu theorem.
\newblock {\em Quarterly Journal of Mathematics}, 57(4):495--504, Dec 2006.

\bibitem[Pet12]{Petridis}
G.~Petridis.
\newblock {New proofs of Pl\"unnecke-type estimates for product sets in
  groups}.
\newblock {\em Combinatorica}, 32:721--733, 2012.

\bibitem[Pr{\'e}71]{Prekopa}
Andr{\'a}s Pr{\'e}kopa.
\newblock Logarithmic concave measures with application to stochastic
  programming.
\newblock {\em Acta Scientiarum Mathematicarum}, 32:301--316, 1971.

\bibitem[Sha19]{MOpost}
George Shakan.
\newblock Nice proof of inequality $(1-x^p)^{1/p}(1-x^q)^{1/q}\ge
  (1-x)(1+x^c)^{1/c}$ where $2^{1/c} = p^{1/p} q^{1/q}$?
\newblock MathOverflow, 2019.
\newblock URL:https://mathoverflow.net/q/343334 (version: 2019-10-08).

\end{thebibliography}

\end{document}